\documentclass[a4paper,10pt,twoside]{amsart}
\usepackage{amsthm}
\usepackage{amsmath}
\usepackage{amssymb}
\usepackage{amsfonts}
\usepackage{amscd}
\usepackage[english]{babel}
\usepackage{stmaryrd}
\usepackage{enumerate}

\newtheorem{thm}{Theorem}

\newtheorem{prop}{Proposition}
\newtheorem{rem}{Remark}

\newtheorem{exam}{Example}

\newtheorem{lem}{Lemma}

\begin{document}

    \title {The matrix equation $XA-AX=f(X)$ \\when $A$ is diagonalizable }
    \author{Gerald BOURGEOIS}
    
    \date{July-12-2014}
    \address{G\'erald Bourgeois, GAATI, Universit\'e de la polyn\'esie fran\c caise, BP 6570, 98702 FAA'A, Tahiti, Polyn\'esie Fran\c caise.}
    \email{bourgeois.gerald@gmail.com}
        

\begin{abstract} 
\end{abstract}

\maketitle
\section{Introduction} 
Let $K$ be an algebraically closed field with characteristic $0$, $A\in M_n(K)$ and $f\in K[x]$. In \cite{1}, we study the matrix equation in the unknown $X\in M_n(K)$
\begin{equation}   \label{fonc}
	XA-AX=f(X).
	\end{equation}
	 We show that necessarily $A$ and any solution $X$ are simultaneously triangularizable. Yet, we study essentially the solutions that have a sole eigenvalue. 
  Here $A\in M_n(K)$ is a diagonalizable matrix and $f\in K[x]$ is such that its roots in $K$ are known. In the first part, we consider all the solutions of Eq (\ref{fonc}) when $A$ has two distinct eigenvalues.
		Clearly, Eq (\ref{fonc}) admit trival solutions, that is to say, solutions that satisfy $AX=XA$ and $f(X)=0$. The following result gives a condition on $A$ and $f$ so that there exist non-trivial solutions.
	\begin{prop}      \label{triang1}
Let $A=\mathrm{diag}(\mu I_p,\lambda I_q)$ where $\lambda,\mu$ are \emph{distinct} elements of $K$ and $\mathcal{T}=f'(f^{-1}(0))\setminus \{0\}$.\\
$i)$ If $\pm(\lambda-\mu)\notin \mathcal{T}$ then any solution $X$ of Eq (\ref{fonc}) satisfies $XA=AX$.\\
 $ii)$ If $\mu-\lambda\notin \mathcal{T}$, then the solutions of Eq (\ref{fonc}) are in the form 
$$X=\begin{pmatrix}P&Q\\0_{q,p}&S\end{pmatrix}\text{ where }f(P)=0_p\;,\;f(Q)=0_q$$
 and there exist such solutions that do not commute with $A$ if and only if 
$$\lambda-\mu\in \mathcal{T}.$$
\end{prop}
In the second part, we study a model of Eq (\ref{fonc}) that admits non-trivial solutions in the sense of Proposition \ref{triang1} $ii)$. For the sake of simplicity, we choose the equation
\begin{equation}   \label{gener} XA-AX=X^2-X^3\text{ in the unknown }X\in M_n(K).  \end{equation}	
where $A\in M_n(K)$ is a \emph{diagonalizable} matrix over $K$.
Yet, the results obtained below appear to be generalizable. We need the following\\
\textbf{Notation.} We can write the complete spectrum $\sigma(A)$ of $A$ in the form 
\begin{equation*} \sigma(A)=\bigcup_{r=1}^k B_r\text{ where the }(B_r)_r\text{ are lists that satisfy the following:} \end{equation*}
$i)$ For every $r$, there exists $\lambda_r\in K$, $c_r\in\mathbb{N}$ such that $B_r=\{L_{r,c_r},\cdots,L_{r,1},L_{r,0}\}$ where, for every $0\leq i\leq c_r$, the list $L_{r,i}$ is composed of the eigenvalue $\lambda_r+i$ numbered with its multiplicity.		\\
$ii)$ If $r\not=s$ and $u\in B_r,v\in B_s$, then $u-v\not=1$.		\\
We consider the ordering of the eigenvalues of $A$ induced by such a sequence $(B_r)_r$ and the associated diagonal form of $A$: there exists an invertible matrix $P$ such that $P^{-1}AP=\bigoplus_{r=1}^kU_r$ where, for every $r$, $U_r=\mathrm{diag}(B_r)$.\\
We show that the solutions of Eq (\ref{gener}) admit a decomposition in direct sum.
\begin{thm}     \label{sol}
Let $A$ be a diagonalizable matrix and let $X$ be a solution of Eq (\ref{gener}). With the previous notation, 
$$P^{-1}XP=\bigoplus_{r=1}^kX_r\text{ where, for every }r,\;\; X_rU_r-U_rX_r={X_r}^2-{X_r}^3.$$
 Moreover, if one adopts the block structure associated with the decomposition of $B_r$, then $X_r$ is a upper triangular block-matrix, the diagonal of which, being
 $$Y_{r,c_r},\cdots,Y_{r,1},Y_{r,0}\text{ and satisfying: for every }i,\;\;{Y_{r,i}}^2-{Y_{r,i}}^3=0.$$
\end{thm}
Finally, we give the dimension of the algebraic variety of solutions of Eq (\ref{gener}) when $A$ satisfies the condition of Proposition \ref{triang1} $ii)$.
\begin{thm}  \label{dim}
	Let $A=\mathrm{diag}(I_p,0_q)$, where $\dfrac{1}{5}\leq\dfrac{p}{q}\leq 5$ and let 
	$$\rho=\left\lfloor{(11p^2+11q^2+2pq)/16}\right\rfloor.$$
	Then the algebraic variety of solutions of Eq (\ref{gener}) has dimension $\rho$ or $\rho-1$.			
				\end{thm} 
	\section{Non-trivial solutions of the matrix equation $XA-AX=f(X)$ }
\subsection{Proof of Proposition \ref{triang1}}		
\begin{proof}
\textbf{Part 1.} $\bullet$ Put $X=\begin{pmatrix}P&Q\\R&S\end{pmatrix}$. Then Eq (\ref{fonc}) can be written 
$$(\lambda-\mu)\begin{pmatrix}0&Q\\-R&0\end{pmatrix}=f(X).$$
The LHS commute with $X$, that implies: 
\begin{equation}  \label{cond1}    QR=0_p\;,\;RQ=0_q\;,\;QS=PQ\;,\;RP=SR.  \end{equation}
$\bullet$ We show that $f(X)=\begin{pmatrix}f(P)&f'(P)Q\\f'(S)R&f(S)\end{pmatrix}$. Indeed, by linearity, it suffices to prove that, for every $k$, $X^k=\begin{pmatrix}P^k&kP^{k-1}Q\\kS^{k-1}R&S^k\end{pmatrix}$. If the previous formula is true, then 
$$X^{k+1}=\begin{pmatrix}P^{k+1}&P^{k-1}(PQ+kQS)\\S^{k-1}(kRP+SR)&S^{k+1}\end{pmatrix}$$
and we conclude by recurrence.\\
$\bullet$ Finally, there are, in addition to relations Eq (\ref{cond1}), the following ones
\begin{equation}  \label{cond2} f(P)=0_p\;,\;f(S)=0_q\;,\;f'(P)Q=(\lambda-\mu)Q\;,\;f'(S)R=(\mu-\lambda)R.  \end{equation}
 We can calculate the solutions in $P,S$. Note that the \emph{sets} of eigenvalues $\sigma(P)$ and $\sigma(S)$ are included in $f^{-1}(0)$. We consider such a couple solution. It remains to solve the linear system
$$(f'(P)\bigotimes I_q)Q=(\lambda-\mu)Q\;,\;(f'(S)\bigotimes I_p)R=(\mu-\lambda)R.$$
The \emph{sets} $\sigma(f'(P)\bigotimes I_q)$ and $\sigma(f'(S)\bigotimes I_p)$ are included in $f'(f^{-1}(0))$. Thus, if $\pm(\lambda-\mu)\notin \mathcal{T}$, then $Q=0,R=0$ and $AX=XA$. In the sequel, we assume that $\mu-\lambda \notin \mathcal{T}$ ; then $R=0$.\\
\textbf{Part 2.} One has $f(X)=\begin{pmatrix}f(P)&Z\\0&f(S)\end{pmatrix}$ with 
$$f'(\begin{pmatrix}P&0\\0&S\end{pmatrix},\begin{pmatrix}0&Q\\0&0\end{pmatrix})=\begin{pmatrix}0&Z\\0&0\end{pmatrix}$$
where $f'(u,v)$, the Frechet derivative of $f$ in $u$, is linear in $v$ (see \cite[Theorem 4.12]{2}). In particular, $Z$ is a linear function of $Q$.
By identification, we obtain 
$$f(P)=0\;,\;f(S)=0\;,\;(\lambda-\mu)Q=Z.$$
 We can calculate the solutions in $P,S$. Note that $\sigma(P)$ and $\sigma(S)$ are included in $f^{-1}(0)$. We consider such a couple solution. It remains to solve a linear equation in the form $\phi(Q)=(\lambda-\mu)Q$ where 
$$\sigma(\phi)=(f[\alpha_i,\beta_j])_{i\leq p;j\leq q},\text{ with  }\sigma(P)=(\alpha_i)_{i\leq p}\text{ and }\sigma(S)=(\beta_j)_{j\leq q}$$
 (see \cite[Theorem 3.9 and the proof of Theorem 3.11]{2}). Here $f[\alpha_i,\beta_j]=0$ if $\alpha_i\not=\beta_j$ and $f'(\alpha_i)$ else. Since, for every $i\leq p,j\leq q$, $f(\alpha_i)=f(\beta_j)=0$, we conclude that 
$$\sigma(\phi)\subset \mathcal{T}\cup \{0\}.$$
Note that, if $\sigma(P)\cap\sigma(S)=\emptyset$, then $Q=0$.
Finally, there is a non-zero solution in $Q$ if and only if $\lambda-\mu\in \mathcal{T}$ ; indeed, if there is $\alpha\in K$ such that $\lambda-\mu=f'(\alpha)\not=0$ and $f(\alpha)=0$, then we choose $P=\alpha I_p\;,\;S=\alpha I_q$ and $Q\in\ker(\phi-(\lambda-\mu)I_{pq})\setminus \{0\}$.
\end{proof}
Note that the previous results remain valid if we change the polynomial $f$ with a holomorphic function.
\begin{exam}
Let $K=\mathbb{C}$, $A=\mathrm{diag}(0_p,I_q)$, $f(x)=\log(x)$, where $\log$ is the principal logarithm ; we seek matrices $X$ that have no eigenvalues on $\mathbb{R}^-$ (the non positive real numbers) such that $XA-AX=\log(X)$. Here $\mathcal{T}=\{1\}$ is equal to $\lambda-\mu$. Thus a solution is in the form $X=\begin{pmatrix}P&Q\\0_{q,p}&S\end{pmatrix}$ where $\log(P)=0_p,\log(S)=0_q$. Finally , the solutions are in the form 
$$X=\begin{pmatrix}I_p&Q\\0&I_q\end{pmatrix}\text{ where }Q\text{ is an arbitrary }p\times q\text{ matrix.}$$
\end{exam}
\begin{exam}
Let $K=\mathbb{C}$, $n=4$, $A=\mathrm{diag}(0_2,I_2)$ ; we consider the equation $XA-AX=\exp(X)-I_4$. Here $f^{-1}(0)=2i\pi\mathbb{Z}$ and $\mathcal{T}=\{1\}$ is $\lambda-\mu$ again. Thus a solution is in the form  
$$X=\begin{pmatrix}P&Q\\0_2&S\end{pmatrix}\text{ where }\exp(P)=I_2\;,\;\exp(S)=I_2\;,\;\phi(Q)=Q.$$
Since $P,S$ are diagonalizable, $\phi$ is diagonalizable too.
Up to a change of basis leaving invariant $A$, we may assume that 
$$P=\mathrm{diag}(2i\pi p_1,2i\pi p_2)\;,\;S=\mathrm{diag}(2i\pi s_1,2i\pi s_2).$$
Note that $Q$ depends on $\mathrm{dim}(\ker(\phi)-I)$ free parameters. We consider the following cases:\\
$\bullet$ If $p_1=p_2=s_1=s_2$, then $\mathrm{dim}(\ker(\phi)-I)=4$ and $Q$ is an arbitrary matrix.\\
$\bullet$ If $p_1=s_1\not=p_2=s_2$, then $\mathrm{dim}(\ker(\phi)-I)=2$ and $Q=\mathrm{diag}(u,v)$ where $u,v$ are arbitrary elements of $K$.\\
$\bullet$ If $p_1=s_1$ is the sole equality, then $\mathrm{dim}(\ker(\phi)-I)=1$ and $Q=\mathrm{diag}(u,0)$ where $u$ is arbitrary.
\end{exam}
\subsection{When $f$ is degenerated function}
For special functions $f$, it can happens that $\{\pm (\lambda-\mu)\}\subset \mathcal{T}$ ; we shall see that there exist solutions that are not in block-triangular form. We fix $P,S$ such that $f(P)=0_p,f(S)=0_q$. Now $Q,R$ satisfy
$$QR=RQ=0\;,\;Q\in\ker(P\bigotimes I-I\bigotimes S^T)\cap\ker(f'(P)\bigotimes I_q-(\lambda-\mu)I_{pq}),$$
$$R\in\ker(S\bigotimes I-I\bigotimes P^T)\cap\ker(f'(S)\bigotimes I_p-(\mu-\lambda)I_{pq}).$$
We are just going to study this typical instance:\\
let $n=4$, $A=\mathrm{diag}(2,2,0,0)$ and $f(x)=x^2-1$ ; here $\mathcal{T}=\{2,-2\}$. We choose $P=S=\mathrm{diag}(1,-1)$. Then the associated solutions are the following ones
$$X=\begin{pmatrix}1&0&0&0\\0&-1&0&u\\v&0&1&0\\0&0&0&-1\end{pmatrix}\text{ where }u\;,\;v\text{ are arbitrary elements of }K.$$ 
In the sequel, we need the following
				\begin{lem}   \label{inva}
				Let $X$ be a solution of the equation 
				\begin{equation}  \label{equaring} XA-AX=X^pg(X)\text{ where }p\geq 2\text{ and }g\text{ is a polynomial}. \end{equation}
$i)$ For every polynomial $P$, $P(X)A-AP(X)=P'(X)X^pg(X)$.\\
				$ii)$ For every integer $k$, $\ker(X^k)$ is $A$-invariant.						
				\end{lem}
				\begin{proof}
				$i)$ We check by induction that 
				$$\text{ for every }i\geq 1\;,\; X^iA-AX^i=iX^{i-1}X^pg(X).$$
				Reasoning by linearity, we deduce the required result.\\
				$ii)$ Let $u$ be such that $X^ku=0$. Then $X^kA-X^{k-1}AX=X^{p+k-1}g(X)$ and $X^kAu=X^{k-1}AXu$. According to $i)$, 
				$$X^{k-1}AXu=(AX^{k-1}+(k-1)X^{p+k-2}g(X))Xu=0.$$				
				\end{proof}
				\section{The matrix equation $XA-AX=X^2-X^3$}
				Now on, we study Eq (\ref{gener}) when $f(x)=x^2-x^3$ and $A$ is diagonalizable ; note that $0$ is a double root of $f$ and the condition of degeneration obtained in Proposition \ref{triang1} $ii)$ is here $\lambda-\mu=-1$. Let $u\in \ker(A-\lambda I)$ and	$X$ be a solution of Eq (\ref{gener}). 
\begin{lem}    \label{rec1}
Let $\lambda=0$. Then, for every integer $s\geq 0$, there is a polynomial $\phi_s$ such that 
$$(A-sI)\cdots(A-I)AXu=\phi_s(X)X^2(I-X)^{s+1}u\text{ and }\phi_s(1)\not=0.$$
\end{lem}
\begin{proof}
We show the first condition by recurrence on $s$. If $s=0$, then $AXu=X^2(X-I)u$ and $\phi_0(X)=-1$. Assume that the result is true for $s-1$. Then, according to Lemma \ref{inva} $ii)$, 
$$(A-sI)\phi_{s-1}(X)X^2(I-X)^su=\phi_{s-1}(X)X^2(I-X)^sAu+\phi_s(X)X^2(I-X)^{s+1}u$$
$$\text{ where }\phi_s(X)=-X^2{\phi_{s-1}}'(X)-2X\phi_{s-1}(X)-s(X+1)\phi_{s-1}(X).$$
It remains to show the second condition. That is equivalent to consider the sequence of polynomials $P_0(x)=1,P_s(x)=x^2{P_{s-1}}'(x)+(sx+2x+s)P_{s-1}(x)$ and to show that, for every $s$, $P_s(1)\not=0$. Clearly, for every $s$, $P_s$ is a polynomial of degree $s$ such that each of its coefficients is positive. Therefore, by an easy recurrence, we obtain that the sequence $(P_s(1))_s$ is increasing.
\end{proof}
	\begin{prop}   \label{decomp}
One has 
$$Xu\in \bigoplus_{0\leq i\leq n-1} \ker(A-(i+\lambda)I).$$								
	\end{prop}
\begin{proof}
We may assume that $\lambda=0$.\\
$\bullet$   Since $X^2-X^3$ is nilpotent, the eigenvalues of $X$ are $0$ or $1$. We show that $X^2(I-X)^n=0$. According to Lemma \ref{inva} $i)$, we may assume that $X=\begin{pmatrix}N&0\\0&L\end{pmatrix}$ where $N$ and $L-I$ are nilpotent and $A=\begin{pmatrix}D&E\\0&F\end{pmatrix}$ where $D$ is diagonalizable over $K$. One has $ND-DN=N^2-N^3$.
According to \cite[Corollary 1]{1}, $ND=DN$ and $N^2(I-N)=0$, that implies $N^2=0$ and we are done.\\
$\bullet$ According to Lemma \ref{rec1}, $(A-(n-1)I)\cdots(A-I)AXu=\phi_{n-1}(X)X^2(I-X)^nu=0$, that is equivalent to the required result.
\end{proof}	
\begin{lem}     \label{poly}
Let $\lambda=0$, $\phi$ be a polynomial such that $\phi(1)\not=0$ and $s,t\geq 0$ be distinct integers. Then there is a polynomial $\psi$ such that $\psi(1)\not=0$ and 
$$(A-sI)\phi(X)X^2(I-X)^tu=\psi(X)X^2(I-X)^tu.$$
\end{lem}
\begin{proof}
As in the proof of Lemma \ref{rec1}, we obtain  
$$\psi(x)=-\phi'(X)X^2(I-X)-2\phi(X)X(I-X)+t\phi(X)X^2-s\phi(X).$$
Thus $\psi(1)=(t-s)\phi(1)\not=0$.
\end{proof}
\begin{prop}    \label{stab}
If $\lambda+s$ is not an eigenvalue of $A$, then
 $$Xu\in  \bigoplus_{0\leq i\leq s-1} \ker(A-(i+\lambda)I).$$
\end{prop}
\begin{proof}
We may assume that $\lambda=0$. Suppose that $(A-(s-1)I)\cdots(A-I)AXu\not=0$ ; according to Lemma \ref{rec1}, this is equivalent to 
$$\phi_{s-1}(X)X^2(I-X)^su\not=0\text{ with }\phi_{s-1}(1)\not=0.$$
 With respect to the matrix $X$, the minimal polynomial of $u$ has the form $X^r(I-X)^{k+1}$ where $k\geq s$ and $r\leq 2$. Then $(A-kI)\cdots(A-I)AXu=\phi_k(X)X^2(I-X)^{k+1}u=0$. Since $A-sI$ is invertible, one has
$$(A-kI)\cdots(A-(s+1)I)\phi_{s-1}(X)X^2(I-X)^su=0.$$
Since $\phi_{s-1}(1)\not=0$, using repeatedly Lemma \ref{poly}, the previous equality can be written 
$$\psi(X)X^2(I-X)^su=0\text{ where }\psi \text{ is a polynomial such that }\psi(1)\not=0.$$
Therefore, the minimal polynomial of $u$ divides $X^2(I-X)^s$, that is contradictory.
\end{proof}		
Now, we can deduce \textbf{Theorem \ref{sol}}.
\begin{proof}
This follows from Proposition \ref{stab} and the construction of the $(B_r)_r$.
\end{proof}
\begin{rem}
To solve Eq (\ref{gener}), it suffices to solve it 
 when $A$ has the form of a matrix $U_r$ with $\lambda_r=0$. 
\end{rem}
 Assume that $A\in M_n(K)$ is non-derogatory and has $k$ distinct eigenvalues ; according to \cite[Theorem 5]{1}, the algebraic variety of the \emph{nilpotent} solutions $X\in M_n(K)$  of Eq (\ref{equaring}) has dimension $n-k$. We look at the dimension of the set of all solutions of Eq (\ref{gener}) in a particular case.
\begin{lem}    \label{nilpo}
The algebraic variety $\{X\in M_n(K)\;|\; X^2=0\}$ has dimension $\left\lfloor{n^2/2}\right\rfloor$.
\end{lem}
\begin{proof}
Any solution $X$ is similar to a matrix in the form $\mathrm{diag}(U_1,\cdots,U_k,0_{n-2k})$, where $U_i=J_2$. We seek the dimension of the similarity class of $X$ as a fuction of $k$. One has $\dim(\mathrm{im}(X))=k$ and $\mathrm{im}(X)\subset\ker(X)$ ; firstly, the choice of $\ker(X)$ depends on $k(n-k)$ parameters ; secondly, the choice of a subspace of dimension $k$ of $\ker(X)$ depends on $k(n-2k)$ parameters ; let $F$ be a complementary of $\ker(X)$ ; finally the choice of an isomorphism $F\rightarrow \mathrm{im}(X)$ depends on $k^2$ parameters. Thus the dimension of the similarity class is $2k(n-k)$. The maximum value of the previous dimension is obtained for $k=\left\lfloor{n/2}\right\rfloor$.
\end{proof}
\begin{lem}   \label{degre3}
The algebraic variety $\{X\in M_n(K)\;|\; X^2-X^3=0\}$ has dimension $\left\lfloor{2n^2/3}\right\rfloor$.
\end{lem}
\begin{proof}
Any solution is similar to a matrix in the form $\mathrm{diag}(U_1,\cdots,U_k,0_t,I_{n-2k-t})$, where $U_i=J_2$. We seek the dimension of the similarity class of $X$ as a function of $k$. Firstly, the choice of the generalized eigenspace of the eigenvalue $0$ depends on $(n-2k-t)(2k+t)$ ; secondly, the choice of the eigenspace of the eigenvalue $1$ depends on $(n-2k-t)(2k+t)$ ; finally, the choice of the restriction of $X$ to the generalized eigenspace of the eigenvalue $0$ depends on $2k(k+t)$ parameters (cf. the proof of Lemma \ref{nilpo}). Thus the dimension of the similarity class is $r_n(k,t)=2n(t+2k)-6k^2-6kt-2t^2$. If $k,t$ were real variables, then a free extremum of $r_n$ is reached for $k=n/3,t=0$ and is $2n^2/3$.  On the boundary, $t=n-2k$ and $r_n=2k(n-k)$ ; the maximum of $r_n$ is reached for $k=n/2$ and is $n^2/2$. Finally, since $k,t$ are integers, the maximum of $r_n$ is reached in a neighborhood of $(n/3,0)$ and is at most $\left\lfloor{2n^2/3}\right\rfloor$. In fact the previous value is reached, for example, when $k=\left\lfloor{n/3}\right\rfloor$ and $t=1$ if $n=2\;\mathrm{mod}\;3$, $t=0$ otherwise.
\end{proof}
	Now we prove \textbf{Theorem \ref{dim}}.		
				\begin{proof}  According to Theorem \ref{sol}, a solution of Eq (\ref{gener}) is in the form $X=\begin{pmatrix}P&Q\\0&S\end{pmatrix}$ where $P\in M_p(K),S\in M_q(K)$, $P^2-P^3=0\;,\;S^2-S^3=0$.\\
				$\bullet$ Case 1. $XA=AX$. According to Lemma \ref{degre3}, a couple solution $(P,S)$ depends at most on $\left\lfloor{2p^2/3}\right\rfloor+\left\lfloor{2q^2/3}\right\rfloor$ parameters that is less than $\rho-1$.\\
				$\bullet$ Case 2. $XA\not=AX$.
				According to the proof of Proposition \ref{triang1}, $\phi(Q)=-Q$ where $\phi+I$ is singular, that is, there are $\lambda\in\sigma(P),\mu\in\sigma(S)$ such that $\lambda=\mu$ and $f'(\lambda)=-1$ ; thus $\sigma(P)$ and $\sigma(S)$ contain each at least once the eigenvalue $1$. Since the eigenvalue $1$ of $P$ or $S$ has index $1$, the eigenvalue $-1$ of $\phi$ has also index $1$. In particular, when $P,S$ are fixed solutions, the dimension $d$ of the vector space of solutions $Q$ is  the number of eigenvalues of $\phi$ equal to $-1$. Finally, $d=\tau_1\tau_2$ where $\tau_1=\dim(\ker(P-I))\;,\;\tau_2=\dim(\ker(S-I))$. We calculate the maximal dimension $\delta(k_1,k_2,\tau_1,\tau_2)$ of the triple $(P,Q,S)$, solutions of Eq (\ref{gener}), such that $P,S$ have Jordan normal forms 
								$$\mathrm{diag}(U_1,\cdots,U_{k_1},0_{p-2k_1-\tau_1},I_{\tau_1}),\mathrm{diag}(U_1,\cdots,U_{k_2},0_{q-2k_2-\tau_2},I_{\tau_2}).$$
			According to the proof of Lemma \ref{degre3}, $r_p(k_1,\tau_1)=2p(p-\tau_1)-2(p-\tau_1)^2+2k_1(p-\tau_1-k_1)$ and $r_q(k_2,\tau_2)=2q(q-\tau_2)-2(q-\tau_2)^2+2k_2(q-\tau_2-k_2)$. Therefore $\delta(k_1,k_2,\tau_1,\tau_2)=r_p(k_1,\tau_1)+r_q(k_2,\tau_2)+\tau_1\tau_2$.	If $k_1,k_2,\tau_1,\tau_2$ were real variable, then a free extremum of $\delta$ is reached in the following point of the boundary
			$$k_1=(5p-q)/16\geq 0\;,\;k_2=(5q-p)/16\geq 0\;,\;\tau_1=(q+3p)/8\;,\;\tau_2=(3q+p)/8$$
			and is $M_{p,q}=(11p^2+11q^2+2pq)/16$. The previous value is also the maximal value of $\delta$ on  the boundary.	Since $k_1,k_2,t_1,t_2$ are integers, the required dimension $\nu$ is at most $\left\lfloor{M_{p,q}}\right\rfloor=\rho$. More precisely $\nu=\rho$ except in the following cases ($\approx 28\%$ of the couples $(p,q)$) where $\nu=\rho-1$.\\
			Below, the values of $p,q$ such that $\rho=\nu-1$ are given $\mod 16$ and $\epsilon=\pm 1$.
			$$p=0\;,\;q=4,5,8,11,12,$$
			$$\epsilon p=1\;,\;\epsilon q=11,12,14,15,$$
			$$\epsilon p=2\;,\;\epsilon q=2,5,6,14,15,$$
			$$\epsilon p=3\;,\;\epsilon q=5,6,8,9,$$
			$$\epsilon p=4\;,\;\epsilon q=0,8,9,12,15,$$				
			$$\epsilon p=5\;,\;\epsilon q=0,2,3,15,$$				
			$$\epsilon p=6\;,\;\epsilon q=2,3,6,9,10,$$				
			$$\epsilon p=7\;,\;\epsilon q=9,10,12,13,$$	
			$$ p=8\;,\; q=0,3,4,12,13.$$			
				\end{proof}	
			\begin{rem}
		$i)$ Note that, if $p=q$, then $\rho=\left\lfloor{3p^2/2}\right\rfloor$. According to the previous table, the required dimension $\nu$ is $\rho$, except if $p=\pm2\text{ or }\pm 6\;\mod\;16$, where $\nu=\rho-1$.	\\
					$ii)$ Clearly, if the inequality $\dfrac{1}{5}\leq \dfrac{p}{q}\leq 5$ is not satisfied, then $\nu$ can be far from $\rho$. For instance, if $p=2,q=26$, then $\nu=\rho-6$.
				\end{rem}


\bibliographystyle{plain}

\end{document}